\documentclass[11pt,a4paper]{amsart}
\usepackage[left=2.5cm,right=2.5cm,bottom=3cm,top=3cm,bindingoffset=0.5cm]{geometry}
\usepackage[utf8]{inputenc}
\usepackage[T1]{fontenc}
\usepackage{color}
\usepackage{amsmath,amssymb,amsthm,amsfonts}
\usepackage{mathtools}
\usepackage{fancyhdr}
\usepackage{csquotes}
\usepackage{bbm}
\usepackage{comment}
\usepackage{tikz}
\usepackage{vmargin}
\usepackage{setspace}
\usepackage{xcolor}
\usepackage{esint}

\definecolor{vertfonce}{rgb}{0.20, 0.46, 0.25}
\definecolor{rougefonce}{rgb}{0.64, 0.09, 0.20}
\usepackage[breaklinks=true,
            colorlinks=true,
            linkcolor=rougefonce,
            citecolor=vertfonce]{hyperref}
\usepackage{cleveref}
\usepackage{mathrsfs, enumerate}

\numberwithin{equation}{section}

\newcommand{\R}{\mathbf{R}}

\newcommand{\dd}{\mathrm{d}}

\newtheorem{theorem}{Theorem}
\newtheorem{lemma}[theorem]{Lemma}
\newtheorem{assumption}{Assumption}
\newtheorem{proposition}[theorem]{Proposition}

\title[Tunneling in homogeneous magnetic field]{Tunneling effect between radial electric wells in a homogeneous magnetic field}
\author[L. Morin]{Léo Morin}
\address[L. Morin]{Department of Mathematics, University of Copenhagen, Universitetsparken 5, DK-2100 Copenhagen \O, Denmark}
\email{lpdm@math.ku.dk }

\begin{document}

\begin{abstract}
We establish a tunneling formula for a Schrödinger operator with symmetric double-well potential and homogeneous magnetic field, in dimension two. Each well is assumed to be radially symmetric and compactly supported. We obtain an asymptotic formula for the difference between the two first eigenvalues of this operator, that is exponentially small in the semiclassical limit. 
\end{abstract}

\maketitle

\section{Introduction}

This article is devoted to the spectral analysis of electromagnetic Schrödinger operators with symmetries. Without magnetic fields, it is known that symmetries of the potential induce tunneling. This translates into a spectral gap between the two first eigenvalues that becomes exponentially small in the semiclassical limit. This effect was studied in \cite{HS,HS85b} where the spectral gap was estimated (see also the books \cite{Hel88,DS}). When adding a homogeneous magnetic field to this model, first general results were obtained in \cite{HS87} (for weak magnetic fields) and \cite{Nakamura} (upper bounds on the spectral gap). The problem was recently reconsidered in \cite{FSW22}, where some upper and lower bounds on the spectral gap are obtained. These bounds were improved in \cite{HK22}, giving a sharp exponential decay rate.

In this paper we consider a potential with two symmetric, radial, and compactly supported wells, as in \cite{FSW22,HK22}. We prove a sharp estimate on the spectral gap in presence of a constant magnetic field. As explained in these two references, the spectral gap is given by an integral term measuring the interaction between the wells, the \emph{hopping coefficient}. However, we suggest here another approach to estimate this coefficient, which is closer to the original spirit of \cite{HS}. This approach provides us with a shorter proof and better estimates.

A similar strategy was recently implemented to prove a purely magnetic tunneling formula, between radial magnetic wells \cite{FMR}. The case of multiple potential wells was also considered in \cite{HS87} (without magnetic field) and recently adapted to the magnetic case in \cite{HKS23}. In these articles it is explained how to reduce the problem to many double-well interactions. Therefore, the result we present below should also have applications to that setting.\\

We consider the following Schrödinger operator acting in the plane $\R^2$,
\begin{equation}\label{defH}
\mathcal H_h = (-ih\nabla - \mathbf A)^2 + V,
\end{equation}
where $V$ is a double-well potential, and $\mathbf A$ is a vector potential generating a uniform magnetic field of strength $B>0$, i.e. $\nabla \times \mathbf A = B$. Without loss of generality, we choose the gauge
\begin{equation}
A(x,y) = (0,Bx).
\end{equation}
The double-well potential is a sum of two disjoint single wells,
\begin{equation}\label{defV}
V(x,y) = v(x-L,y) + v(x+L,y),
\end{equation}
separated by a distance $2L>0$. On the single well we make the following assumptions.\\

\begin{assumption} \label{assump} We assume that
\begin{itemize}
\item $v$ is smooth, compactly supported, non-positive, and radial.
\item $v$ admits a unique minimum $v_0 <0$ reached at $0$, and it is non-degenerate.
\end{itemize}
\end{assumption}
Let $a>0$ denote the radius of the support of $v$, i.e. the smallest positive number such that $\mathrm{supp}(v) \subset \overline{B}(0,a)$. We exploit the radiality of the single well problem. We denote by $\varphi$ the ground state of the single well Hamiltonian,
\[ \mathcal H_h^{sw} = (-ih\nabla - \mathbf A^{\rm{rad}})^2 + v, \]
in radial gauge $\mathbf A^{\rm{rad}}(x,y) = \frac{B}{2}(-y,x)$. Then $\varphi$ is radial, as explained in \cite[Section 2.2]{HK22}, and the function $u(|X|) = \varphi(X)$ satisfies the equation
\begin{equation}
- h^2 \frac 1 r \partial_r (r \partial_r u) + \Big( \frac{B^2 r^2}{4} + v(r) \Big) u = \mu_h u.
\end{equation}
where $\mu_h$ is the eigenvalue of $\varphi$. Hence, the study of the single-well problem is reduced to a radial Schrödinger equation with effective potential $v_B(r) = \frac{B^2 r^2}{4} + v(r)$. As recalled in \cite{HK22}, this is a very standard setting. The decay of $\varphi$ can be optimally measured in terms of the Agmon distance,
\begin{equation}\label{eq:agm}
d(r_1,r_2) = \int_{r_1}^{r_2} \sqrt{\frac{B^2 r^2}{4} + v(r) - v_0} \dd r.
\end{equation}
Moreover, we have explicit WKB approximations for $\varphi$ (see Lemma \ref{lem:WKB} below). However, the double well problem \emph{is not} directly reduced to a Schrödinger equation with potential $v_B$. If that was the case, the spectral gap would have an exponential rate of decay given by $S_0 = 2d(0,L)$, as in the non-magnetic situation \cite{HS}. Here, consistently with \cite{HK22}, we prove that the spectral gap is much smaller.
 
\begin{theorem}\label{thm}
Let $v$ be a single well potential satisfying Assumption \ref{assump}, and let $\mathcal H_h$ be as in \eqref{defH}, with $V$ given by \eqref{defV}. Also assume that $L > (1+\frac{\sqrt{3}}{2})a$.
Then the spectral gap between the two smallest eigenvalues of $\mathcal H_h$ satisfies
\[ \lambda_2(h) - \lambda_1(h) = 2C(L,B,v) h^{\frac 1 2} e^{-\frac S h } (1+o(1))  \]
as $h \to 0$, for some constant $C(L,B,v) >0$, and with
\begin{equation}\label{eq:defS}
S = 2d(0,L) + \int_0^L \sqrt{\frac{B^2(2L-r)^2}{4} -v_0} - \sqrt{\frac{B^2r^2}{4} - v_0} \dd r. 
\end{equation}
The constant $C(L,B,v)$ can be computed explicitly, even though we have no simple interpretation, see \eqref{eq.249}.
\end{theorem}

Theorem \ref{thm} is the first result establishing an asymptotic formula for the spectral gap of $\mathcal H_h$. First results in this direction appeared in a general framework in \cite{HS87}, when the magnetic field is weak. Then first bounds on $S$ were obtained in \cite{Nakamura,FSW22}. Recently the sharp exponential decay rate \eqref{eq:defS} was found in \cite{HK22}, where it is proven that
\[h \ln (\lambda_2(h) - \lambda_1(h)) \sim -S.\]  Another related problem is the tunneling between purely magnetic wells, when there is no potential and the confinement is generated by a non-homogeneous magnetic field. This problem was studied in \cite{FMR} in the case of radial wells, using a similar strategy. Finally, other magnetic tunneling estimates have recently been established in the case of Neumann boundary conditions \cite{BHR22}, vanishing magnetic fields \cite{AA23}, or discontinuous magnetic fields \cite{FHK22}. See also \cite{BHR17} in a one-dimensional setting.\\

\textit{Remark 1.} The additional integral term in $S$ is a purely magnetic effect, with no analogue in the non-magnetic case, thus making this model especially interesting. It is directly related to oscillations of the eigenfunctions. Indeed, the eigenfunctions $\varphi_\ell$ and $\varphi_r$ generated by the left and right well respectively are related by a magnetic translation (see \eqref{eq:leftwell}),
\begin{equation}
\varphi_\ell(x,y) = e^{\frac i h BL y} \varphi_r(x+2L,y).
\end{equation}
Thus, they do not oscillate at the same frequency and the relative oscillation is fast as $h \to 0$. This rapidly oscillating phase appears in the hopping coefficient, making it smaller than what one could expect (see Section \ref{sec:hopping}).\\


\textit{Remark 2.} The condition $L > (1+\frac{\sqrt 3}{2})a$ is a limit of our strategy, which relies entirely on the radiality of the single well. Indeed, the problem centered on the left well is only radial up to distance $2L-a$, and not $2L$. This slight difference makes our decay estimates on the eigenfunctions non-optimal (see Lemma \ref{lem:approx}). To overcome this problem, we would need to understand very precisely the non-radial situation.\\

\textit{Remark 3.} We recover the bounds on $S$ stated in \cite{FSW22,HK22} as follows. The result \cite[Theorem 1.2]{HK22} can be translated as
\begin{equation}\label{eq:boundfromHK}
d(0,2L-a) + d(0,a) \leq S \leq d(0,2L),
\end{equation} 
which follows from
\begin{equation}
S- d(0,2L) = \int_0^L \sqrt{\frac{B^2r^2}{4} + v - v_0} - \sqrt{\frac{B^2r^2}{4} - v_0} \dd r \leq 0,
\end{equation}
and
\begin{equation}
S-  \big( d(0,2L-a) + d(0,a) \big) = \int_0^a \sqrt{\frac{B^2 (2L-r)^2}{4}-v_0} - \sqrt{\frac{B^2r^2}{4} - v_0} \dd r \geq 0.
\end{equation}
Moreover, the main result from \cite{FSW22} is
\begin{equation}
BL^2 - BLa \leq S \leq BL^2 + 2 \sqrt{|v_0|} L + \gamma_0,
\end{equation}
which follows from \eqref{eq:boundfromHK} with $\gamma_0 = \int_0^a \sqrt{v-v_0} \dd r$.
\\

\textit{Strategy.} In Section \ref{sec:sw} we describe the single well problem. Most of this section is standard since the ground state solves a radial Schrödinger equation without magnetic field. In Section \ref{eq:dw} we prove that the spectral gap is given by the hopping coefficient. Our proof is inspired by the non-magnetic situation in \cite{DS,HS,Hel88}, and is somewhat simpler that the one of \cite{FSW22,HK22}. Finally, in Section \ref{sec:hopping} we estimate the hopping coefficient, thus proving Theorem \ref{thm}.

\section{The single well problem}\label{sec:sw}

Let $\varphi$ be the normalized ground state for the single well problem in radial gauge,
\begin{equation}\label{eq:singlewell}
(-ih\nabla - \mathbf A^{\rm{rad}})^2 \varphi + v \varphi = \mu_h
 \varphi,
\end{equation}
with 
\begin{equation}
\mathbf A^{\rm{rad}}(x,y) = \frac B 2 (-y,x).
\end{equation}
We first collect the following basic facts about $\varphi$.
\begin{proposition}\label{prop:singlewell}
The ground state energy $\mu_h$ of $(-ih\nabla - \mathbf A^{\rm{rad}})^2$ satisfies
\[ \mu_h = v_0 + h \sqrt{B^2 + 2 v''(0)} + \mathcal O(h^{2}). \]
The associated eigenfunction $\varphi$ is radial and real-valued. Moreover, the first excited eigenvalue $\mu_{h,1}$ satisfies
\[ \mu_{h,1} = v_0 + 3h\sqrt{B^2 + 2 v''(0)} + \mathcal O(h^{2}). \]
\end{proposition}

A detailed proof of Proposition \ref{prop:singlewell} is given in \cite[Section 2.3]{HK22}. It follows from a harmonic approximation of $v$ near its minimum. Indeed, when $v$ is quadratic, the problem becomes more explicit and we can prove that the minimizer is radial. The harmonic approximation gives radiality when $v$ has a non-degenerate minimium. The function $\varphi$ being radial, equation \eqref{eq:singlewell} can be rewritten in polar coordinates, using the notation $\varphi (X) = u(|X|)$, 
\begin{equation}\label{eq:phirad}
-h^2 \frac 1 r \partial_r \big( r \partial_r u \big) + \frac{B^2r^2}{4}u + vu = \mu_h u.
\end{equation}
This is a radial Schrödinger equation, with effective potential $v_B(r) = \frac{B^2r^2}{4} + v(r)$. From this observation you deduce Proposition \ref{prop:singlewell}, by harmonic approximation of $v_B$ near its minimum (see \cite[Chapter 2]{Hel88} for instance). Moreover, the ground state $\varphi$ has the following WKB approximation (\cite[Theorem 2.3.1]{Hel88}).

\begin{proposition}\label{lem:WKB0}
The ground state $\varphi$ of the single well problem has the following WKB approximation,
\[ \Big| e^{\frac{d(0,|X|)}{h}}\varphi(X) - h^{-\frac 1 2} K(|X|) \Big| = \mathcal O(h^{\frac 1 2}),\]
uniformly on any compact, where $d$ is the Agmon distance introduced in \eqref{eq:agm}, and
\begin{equation}
K(r) = \sqrt{\frac{B^2+2v''(0)}{4\pi}} \exp \Big( - \int_0^{r} \frac{v_B'(s)}{4(v_B(s)-v_0)} + \frac{1}{2s} - \frac{\sqrt{B^2 + 2 v''(0)}}{2\sqrt{v_B(s)-v_0}} \dd s \Big) .
\end{equation}
\end{proposition}


As observed in \cite[Equation (2.9)]{FSW22}, the function $\varphi$ also has an explicit integral formula outside the support of $v$, since equation \eqref{eq:phirad} is related to some special functions.

\begin{lemma}\label{lem:WKB}
The normalized ground state $\varphi$ of the single well problem, solution of \eqref{eq:singlewell}, is radial. Moreover, it has the following integral formula for $|X| >a$,
\begin{equation}\label{eq:integral}
\varphi(X) = C_h \exp \Big( - \frac{B |X|^2}{4h} \Big) \int_0^\infty \exp \Big( - \frac{B|X|^2 t}{2h} \Big) t^{\alpha-1} (1+t)^{-\alpha} \dd t,
\end{equation}
where $\alpha = \frac 1 2 - \frac{\mu_h}{2Bh}$, and $C_h$ is a normalization constant. Moreover, 
\begin{equation}\label{eq:Ch}
C_h \sim h^{-1} K(L) \sqrt{\frac{f''(t_L)}{2\pi}} t_L^{1-\nu}(1+t_L)^{\nu} e^{\frac{\tilde{d}(L) -  d(0,L)}{h}} \quad \text{as} \quad h \to 0,
\end{equation}
where $\tilde d(L) = \int_0^L \sqrt{\frac{B^2r^2}{4} - v_0} \dd r$, the Agmon distance $d(0,L)$ was defined in \eqref{eq:agm}, $t_L$ and $f''(t_L)$ are given in \eqref{eq.tL0}, and $\nu = \frac 1 2 -  \frac{\sqrt{B^2 + 2v''(0)}}{2B}$.
\end{lemma}

\begin{proof}
Note that $\varphi$ satisfies equation \eqref{eq:phirad} which can be solved by special functions outside the support of $v$ (It is related to Kummer functions, see \cite{FSW22}). To estimate $C_h$, we combine the WKB approximation from Proposition \ref{lem:WKB0},
\begin{equation}\label{eq:varphia}
\varphi(L) = h^{- \frac 1 2} K(L) e^{-\frac{d(0,L)}{h}}(1+o(1)),
\end{equation}
with an estimate of \eqref{eq:integral} using the Laplace method. Indeed, using that
\begin{equation}\label{eq.alpha0}
\alpha = \frac{|v_0|}{2Bh} + \nu + o(1),
\end{equation}
we find
\begin{equation}
\varphi(L) = C_h \int_0^\infty \exp \Big( - \frac{f(t)}{h} \Big) t^{\nu-1}(1+t)^{-\nu} \dd t \, \big( 1+  o(1)\big),
\end{equation}
with
\begin{equation}
f(t) = \frac{B L^2}{4}(1+2 t) + \frac{|v_0|}{2B} \ln \Big( \frac{1+t}{t} \Big).
\end{equation}
The first and second derivatives of $f$ are
\begin{equation}
f'(t) = \frac{BL^2}{2} - \frac{|v_0|}{2B} \frac{1}{t(1+t)}, \qquad f''(t) = \frac{|v_0|}{2B} \frac{2t+1}{t^2(1+t)^2}.
\end{equation}
In particular, $f$ has a unique critical point $t_L>0$, which is also a global minimum,
\begin{equation}\label{eq.tL0}
t_L = \frac 1 2 \Big( -1 + \sqrt{1 + \frac{4 |v_0|}{B^2 L^2}} \, \Big), \quad \text{with} \quad f''(t_L) = \frac{B^2L^3}{2|v_0|} \sqrt{B^2 L^2 + 4 |v_0|}.
\end{equation}
By the Laplace method, we deduce
\begin{equation}\label{eq.phiL3}
\varphi(L) = C_h \sqrt{\frac{2\pi h}{f''(t_L)}} e^{- \frac{f(t_L)}{h}} t_L^{\nu -1} (1+t_L)^{-\nu} \big( 1 + o(1) \big).
\end{equation}
Also note that $f(t_L) = \tilde d(L)$. 
We combine this with \eqref{eq:varphia} to get the estimate on $C_h$.
\end{proof}


\section{The double well problem}\label{eq:dw}

In order to study the double well problem, we follow the strategy of \cite{HS} and compute the matrix of $\mathcal H_h$ in a basis $(\varphi_\ell,\varphi_r)$, where $\varphi_\ell$ (resp. $\varphi_r$) is the ground state generated by the left well (resp. the right well). More precisely, we define $\varphi_\ell$ and $\varphi_r$ as the normalized solutions to
\begin{align}
(-ih\nabla - \mathbf A)^2 \varphi_\ell + v(x+L,y)\varphi_\ell = \mu_h \varphi_\ell, \label{eq:leftwell}\\
(-ih\nabla - \mathbf A)^2 \varphi_r + v(x-L,y)\varphi_r = \mu_h \varphi_r, \label{eq:rightwell}
\end{align}
respectively. Note that $\varphi_\ell$ and $\varphi_r$ are related to the radial solution $\varphi$ of the single-well problem \eqref{eq:singlewell} as follows. Let us denote by $\mathbf A^\ell$ and $\mathbf A^r$ the radial gauges centered at $(-L,0)$ and $(L,0)$ respectively, namely
\begin{align*}
\mathbf A^\ell(x,y) = \frac B 2 (-y, x+L), \qquad \mathbf A^r(x,y) = \frac B 2 (-y,x-L). 
\end{align*}
We define the two functions $\sigma_\ell$ and $\sigma_r$ by
\begin{equation}
\sigma_\ell(x,y) = \frac B  2 y (L-x), \qquad \sigma_r(x,y) = - \frac{B}{2} y (L+x).
\end{equation}
They satisfy
\begin{equation}
\nabla \sigma_\ell = \mathbf A^\ell - \mathbf A, \qquad \nabla \sigma_r = \mathbf A^r - \mathbf A, \qquad \text{and} \qquad \sigma_\ell - \sigma_r = BL y.
\end{equation}
Then $\varphi_\ell$ and $\varphi_r$ are related to $\varphi$ by a magnetic translation,
\begin{align}
\varphi_\ell(x,y) = e^{-\frac{i \sigma_\ell}{h}} \varphi(x+L,y), \qquad
\varphi_r(x,y) = e^{-\frac{i \sigma_r}{h}} \varphi(x-L,y).
\end{align}
The difference between the two smallest eigenvalues of $\mathcal H_h$ can be estimated using $\varphi_\ell$ and $\varphi_r$ through the hopping coefficient, as stated in the following theorem, which can also be found in \cite[Section 4]{FSW22}.

\begin{theorem}\label{thm.split}
The two smallest eigenvalues of $\mathcal H_h$ satisfy
\begin{align*}
 \lambda_1 &= \mu_h - |w_h| +   \mathcal O \big( h^{-3} e^{-\frac{2d(0,2L-a)}{h}} \big), \\
\lambda_2 &= \mu_h + |w_h| +   \mathcal O \big( h^{-3} e^{-\frac{2d(0,2L-a)}{h}} \big), 
\end{align*}
where 
\begin{equation}
w_h = \int_{\R^2} v(x+L,y) \varphi_\ell(x,y) \overline{\varphi_r(x,y)} \dd x \dd y.
\end{equation}
\end{theorem}
The proof of Theorem \ref{thm.split} is a standard application of the Helffer-Sjöstrand strategy (see \cite[Theorem 3.9]{HS} or \cite{FMR,Hel88}). For the reader's convenience, we recall the main ideas in Sections \ref{sec.approx} and \ref{sec.proofthm2} below.

\subsection{An approximation lemma}\label{sec.approx}

First of all, the Agmon estimates give exponential localization of the eigenfunctions of $\mathcal H_h$ inside the wells. This localization is enough to prove that the spectrum of the double-well operator is the superposition of the spectra of the one-well operators, modulo $\mathcal O(h^\infty)$ (as in \cite{HS,Hel88}). The proof of this standard result is omitted.

\begin{lemma}\label{lem.superposition}
There exists $c,h_0>0$ such that, for $h \in (0,h_0)$,
\[ | \lambda_1(h) - \mu_h | = \mathcal O(h^\infty), \quad  | \lambda_2(h) - \mu_h | = \mathcal O(h^\infty) ,\]
and $\lambda_3(h) - \lambda_1(h) \geq ch$.
\end{lemma}

Let $\Psi_1$, $\Psi_2$ be the two first eigenfunctions of $\mathcal H_h$, and $\Pi$ the spectral projector on $\mathrm{Ran}( \Psi_1,\Psi_2)$. The lemma below shows that the single well states $\varphi_\ell$ and $\varphi_r$ are close to this eigenspace.

\begin{lemma}\label{lem:approx} There exists a $C>0$ such that, for $h$ small enough and $j = \ell, r$ we have
\begin{equation*}
\| \varphi_j - \Pi \varphi_j \| \leq C h^{-\frac 3 2} e^{- \frac{d(0,2L-a)}{h}}, \quad \| (-ih\nabla - \mathbf A) (\varphi_j - \Pi \varphi_j) \| \leq Ch^{-\frac 32} e^{-\frac{d(0,2L-a)}{h}}.
\end{equation*}
\end{lemma}

\begin{proof}
We focus on $\varphi_\ell$, the estimates on $\varphi_r$ being identical. Let $\lambda_3$ be the third eigenvalue of $\mathcal H_h$. By definition of $\Pi$, we have the lower bound
\begin{equation}\label{eq:lem10}
\lambda_3 \| (I - \Pi)\varphi_\ell \|^2 \leq \langle \mathcal H_h (I-\Pi) \varphi_\ell, (I-\Pi) \varphi_\ell \rangle.
\end{equation}
On the other hand, since $\mathcal H_h$ and $\Pi$ commute, we can use the eigenvalue equation \eqref{eq:leftwell} to get
\begin{equation} \label{eq:lem11}
\langle \mathcal H_h (I-\Pi) \varphi_\ell, (I-\Pi) \varphi_\ell \rangle = \mu_h \| (I-\Pi) \varphi_\ell \|^2 + \langle (I-\Pi) v_r \varphi_\ell, (I-\Pi) \varphi_\ell \rangle,
\end{equation}
with $v_r(x,y) = v(x-L,y)$. Combining \eqref{eq:lem10} and \eqref{eq:lem11} we find
\begin{equation}\label{eq:lem12}
\| (I - \Pi) \varphi_\ell \| \leq (\lambda_3 - \mu_h)^{-1} \| v_r \varphi_\ell \|.
\end{equation}
For the gradient estimate we start from
\begin{equation}
\| (-ih\nabla -\mathbf A) (I-\Pi) \varphi_\ell \|^2 \leq \langle \mathcal H_h (I-\Pi) \varphi_\ell, (I-\Pi) \varphi_\ell \rangle + 2 |v_0| \| (I-\Pi) \varphi_\ell \|^2.
\end{equation}
We then use \eqref{eq:lem11} and \eqref{eq:lem12} to deduce
\begin{equation}\label{eq:lem13}
\| (-ih\nabla - \mathbf A) (I-\Pi) \varphi_\ell \|^2 \leq \big((\lambda_3 - \mu_h)^{-1} +  (\mu_h + 2 |v_0|) (\lambda_3 - \mu_h)^{-2} \big) \| v_r \varphi_\ell \|^2.
\end{equation}
With Proposition \ref{lem:WKB0} we finally bound $\| v_r \varphi_\ell \|$ in \eqref{eq:lem12} and \eqref{eq:lem13},
\begin{equation}
\| v_r \varphi_\ell \| \leq C h^{-\frac 12}e^{- \frac{d(0,2L-a)}{h}}.
\end{equation}
The result follows since $\lambda_3 - \mu_h \geq c h$, by Lemma \ref{lem.superposition}.
\end{proof}


\subsection{Proof of Theorem \ref{thm.split}} \label{sec.proofthm2}

From Lemma \ref{lem:approx} we deduce, for $i,j \in \lbrace \ell, r \rbrace$, with $\psi_j = \Pi \varphi_j$,
\begin{align}\label{eq:psipsi}
\langle \psi_i | \psi_j \rangle &= \langle \varphi_i | \varphi_j \rangle + \mathcal O \big( h^{-3} e^{- \frac{2d(0,2L-a)}{h}} \big), \\
\langle \psi_i | \mathcal H_h \psi_j \rangle &= \langle \varphi_i | \mathcal H_h \varphi_j \rangle + \mathcal O \big( h^{-3} e^{-\frac{2d(0,2L-a)}{h}} \big),\label{eq:psiHpsi}
\end{align}
and also
\begin{align}
\langle \varphi_\ell | \mathcal H_h \varphi_\ell \rangle &= \mu_h + \mathcal O \big( h^{-3} e^{-\frac{2d(0,2L-a)}{h}} \big), \label{eq:phiellphiell} \\
\langle \varphi_r | \mathcal H_h \varphi_r \rangle &= \mu_h + \mathcal O \big( h^{-3} e^{-\frac{2d(0,2L-a)}{h}} \big),\label{eq:phirphir} \\
\langle \varphi_\ell | \mathcal H_h \varphi_r \rangle &= \mu_h \langle \varphi_\ell | \varphi_r \rangle + w_h.\label{eq:phiellphir}
\end{align}
We now use the orthonormal basis
\begin{equation}\label{eq.defpsihat}
\hat \psi_\ell = \frac{1}{\| \psi_\ell \|} \psi_\ell, \qquad \hat \psi_r = \frac{\psi_r - \langle \psi_r, \hat \psi_\ell \rangle \hat \psi_\ell}{\| \psi_r - \langle \psi_r, \hat \psi_\ell \rangle \hat \psi_\ell \|}.
\end{equation}
It follows from \eqref{eq:psipsi}, \eqref{eq:psiHpsi}, \eqref{eq:phiellphiell}, \eqref{eq:phirphir} and \eqref{eq:phiellphir} that the matrix of $\mathcal H_h$ in the orthonormal basis $(\hat \psi_\ell, \hat \psi_r)$ is
\begin{equation}
\begin{pmatrix}
\langle \hat \psi_\ell |\mathcal H_h \hat \psi_\ell \rangle & \langle \hat \psi_\ell | \mathcal H_h \hat \psi_r \rangle \\
\langle \hat \psi_r | \mathcal H_h \hat \psi_\ell \rangle & \langle \hat \psi_r | \mathcal H_h \hat \psi_r \rangle
\end{pmatrix}
=
\begin{pmatrix}
\mu_h & w_h \\
\overline {w}_h & \mu_h
\end{pmatrix}
+  \mathcal O \big( h^{-3} e^{-\frac{2d(0,2L-a)}{h}} \big).
\end{equation}
We deduce that the two first eigenvalues of $\mathcal H_h$ are
\begin{equation}
\lambda_{\pm} = \mu_h \pm | w_h | +   \mathcal O \big( h^{-3} e^{-\frac{2d(0,2L-a)}{h}} \big),
\end{equation}
and Theorem \ref{thm.split} follows. Note that, instead of $\hat \psi_j$, one could also use a more symmetric orthonormalization as in \cite{FMR}. The choice \eqref{eq.defpsihat} is the same as in \cite{FSW22}.

\section{Estimates on the hopping coefficient $w_h$}\label{sec:hopping}

We prove here the following estimate on the hopping coefficient, which governs the spectral gap by Theorem \ref{thm.split}.

\begin{theorem}\label{thm.wh}
There exists a constant $C(B,L,v) >0$ such that
\[ w_h = - C(B,L,v) h^{\frac 1 2} e^{- \frac S h}(1+o(1))\]
as $h \to 0$, with 
\begin{equation}\label{form.S}
S = 2d(0,L) + \int_0^L \sqrt{\frac{B^2(2L-r)^2}{4} -v_0} - \sqrt{\frac{B^2r^2}{4} - v_0} \dd r.
\end{equation}
The constant $C(B,L,v)$ has a long but explicit expression, see \eqref{eq.249}.
\end{theorem}

Contrary to the approach of  \cite{FSW22}, we use the representation of $w_h$ as an explicit integral on a line separating the two wells, as in \cite[Equation (2.25)]{HS}. Compared to the non-magnetic case, the novelty here is that the phase appearing in this integral takes complex values due to the magnetic flux. A similar effect was observed in the purely magnetic situation \cite{FMR}. However for constant magnetic fields the estimate is simpler, since the resulting complex integral is Gaussian. We give the details of the proof in Section \ref{subsec.wh} below.

Our main result, Theorem \ref{thm}, follows from the reduction to the hopping coefficient Theorem \ref{thm.split}, together with Theorem \ref{thm.wh}. We only need to ensure that the error from Theorem \ref{thm.split} is smaller than $w_h$. We discuss this condition in Section \ref{subsec.fin}.

\subsection{Proof of Theorem \ref{thm.wh}}\label{subsec.wh}
Since $v$ is supported in the ball of radius $a$, and $L>a$, the integral defining $w_h$ can be restricted to the left half-plane, $\Omega_\ell = \lbrace (x,y) | x <0 \rbrace$,
\begin{equation}
w_h = \int_{\Omega_\ell} v(x+L,y) \varphi_\ell \overline{\varphi_r} \dd x \dd y.
\end{equation}
Using the equations \eqref{eq:leftwell} and \eqref{eq:rightwell} satisfied by $\varphi_\ell$ and $\varphi_r$ on $\Omega_\ell$, we find
\begin{equation}
w_h = \int_{\Omega_\ell} \varphi_\ell \cdot \overline{(-ih \nabla - \mathbf A)^2 \varphi_r} \dd x \dd y - \int_{\Omega_\ell} (-ih \nabla - \mathbf A)^2 \varphi_\ell \cdot \overline{\varphi_r} \dd x \dd y.
\end{equation}
A partial integration yields
\begin{equation}
w_h = ih \int_{\partial \Omega_\ell} \varphi_\ell \overline{ (-ih\nabla - \mathbf A) \varphi_r} \cdot \mathbf n + \overline{\varphi_r} (-ih \nabla - \mathbf A) \varphi_\ell \cdot \mathbf n,
\end{equation}
where $\mathbf n = (1,0)$ is the outward normal to $\Omega_\ell$. Now due to our choice of gauge, $A_1 = 0$ and only remains
\begin{equation}\label{eq:wnice}
w_h = h^2 \int_{\R}  \overline{\varphi_r(0,y)} \partial_x \varphi_\ell (0,y) -  \varphi_\ell (0,y) \overline{\partial_x \varphi_r(0,y) } \dd y.
\end{equation}
We recall that $\varphi_\ell$ and $\varphi_r$ are related to the radial solution $\varphi$ of the single-well problem by
\begin{align}
\varphi_\ell(x,y) = e^{-\frac{i \sigma_\ell}{h}} \varphi(x+L,y), \qquad
\varphi_r(x,y) = e^{-\frac{i \sigma_r}{h}} \varphi(x-L,y).
\end{align}
We now use these relations, together with the integral representation formula in Lemma~\ref{eq:integral}, to calculate $w_h$. First of all,
\begin{equation}
\varphi_j(0,y) = C_h e^{-\frac{i \sigma_j}{h}} \int_0^\infty e^{- \frac{ B (L^2 + y^2)}{4h} (1+2t)} t^{\alpha -1} (1+t)^{-\alpha} \dd t,
\end{equation}
and 
\begin{align*}
\partial_x \varphi_\ell (0,y) &= C_h e^{-\frac{i\sigma_\ell}{h}}  \int_0^\infty e^{- \frac{ B (L^2 + y^2)}{4h} (1+2t)} \Big(- \frac i h \partial_x \sigma_\ell - \frac{B L}{2h} (1+2t) \Big)  t^{\alpha -1} (1+t)^{-\alpha} \dd t,\\
\partial_x \varphi_r(0,y) &= C_h e^{-\frac{i\sigma_r}{h}} \int_0^\infty e^{- \frac{ B (L^2 + y^2)}{4h} (1+2t)} \Big(- \frac i h \partial_x \sigma_r + \frac{B L}{2h} (1+2t) \Big)  t^{\alpha -1} (1+t)^{-\alpha} \dd t.
\end{align*}
We insert this in \eqref{eq:wnice},
\begin{align}
w_h = -h^2 C_h^2 &\int_{\R \times \R_+^2} \Big( \frac i h (\partial_x \sigma_r + \partial_x \sigma_\ell) + \frac{BL}{h} (1+t+s) \Big) e^{\frac i h ( \sigma_r - \sigma_\ell )} \nonumber \\
&e^{ - \frac{ B (L^2 + y^2)}{2h} (1+t+s) } (ts)^{\alpha -1} (1+t)^{-\alpha}(1+s)^{-\alpha}  \dd y \dd t \dd s. \label{eq:wint1}
\end{align}
In \eqref{eq:wint1} we use $\sigma_r - \sigma_\ell = -BLy$,
\begin{equation*}
w_h = -h C_h^2 BL \int_{\R \times \R_+^2} e^{- \frac{B(L^2+y^2)}{2h}(1+t+s) - \frac{iBLy}{h}} \frac{(1+t+s -i  y / L) (ts)^{\alpha - 1}}{(1+t)^{\alpha} (1+s)^{\alpha}} \dd y \dd t \dd s.
\end{equation*}
Here the $y$-integral is Gaussian: inside the exponential we have
{\small
\begin{equation*}
\frac{B(L^2 + y^2)}{2}(1+t+s) + i BL y = \frac{B(1+t+s)}{2} \big( y + \frac{iL}{1+t+s} \big)^2 + \frac{BL^2}{2} \big( 1+ t +s  + \frac{1}{1+t+s} \big),
\end{equation*}}
and this complex-centered Gaussian is integrated as
\begin{equation}
\int_\R \exp \Big(- \frac{B(1+t+s)}{2h} \big( y +\frac{iL}{1+t+s}  \big)^2 \Big) \dd y = \sqrt{\frac{2\pi h}{B(1+t+s)}},
\end{equation}
and
\begin{equation}
\int_\R \frac{-iy}{L}\exp \Big(- \frac{B(1+t+s)}{2h} \big( y +\frac{iL}{1+t+s}  \big)^2 \Big) \dd y = \frac{-1}{1+t+s}\sqrt{\frac{2\pi h}{B(1+t+s)}}.
\end{equation}
Thus,
\begin{equation*}
w_h = - h^{\frac 3 2} C_h^2 \sqrt{2 \pi B L^2 } \int_{\R_+^2} \exp \Big( - \frac{BL^2}{2h} \big( 1+t+s + \frac{1}{1+t+s} \big) \Big) \frac{\omega(s,t)(ts)^{\alpha-1}}{(1+t)^\alpha (1+s)^\alpha} \dd t \dd s,
\end{equation*}
with $\omega(s,t) = (1+t+s)^{1/2} - (1+t+s)^{-3/2}\geq 0$.
We replace $\alpha = \frac{|v_0|}{2Bh} + \nu + o(1)$ as $h\to 0$, 
\begin{equation}\label{eq:wexplicit}
w_h= -h^{\frac 3 2} C_h^2  \sqrt{2 \pi B L^2 } \int_{\R_+^2} e^{- \frac{g(s,t)}{h}} \frac{\omega(s,t)(ts)^{\nu-1}}{(1+t)^\nu (1+s)^\nu} \dd t \dd s \, (1+o(1)),
\end{equation}
with 
\begin{equation}
g(s,t) =\frac{ BL^2}{2} \Big( 1+t+s + \frac{1}{1+t+s} \Big) + \frac{|v_0|}{2B} \Big( \ln \big( \frac{1+t}{t} \big) + \ln \big( \frac{1+s}{s} \big) \Big).
\end{equation}
The function $g$ has a unique critical point, which is also a global minimum, at $t=s=t_\star$, with
\begin{equation}
t_\star = \frac 1 2 \sqrt N - \frac 1 2 + \frac 1 2 \sqrt{1+N}, \qquad N = \frac{|v_0|}{B^2L^2}.
\end{equation}
Moreover,
\begin{equation}
g(t_\star,t_\star) = BL^2 \Big( \sqrt{1+N} + N \ln \big( \frac{1+\sqrt{1+N}}{\sqrt N}\big) \Big) =  \int_0^{2L} \sqrt{\frac{B^2 r^2}{4}+ |v_0|} \dd r.
\end{equation}
Using the Laplace method, the integral \eqref{eq:wexplicit} can be estimated as
\begin{equation}\label{eq.249}
w_h \sim - h^{5/2} C_h^2 \frac{(2\pi)^{3/2}\sqrt{BL^2}}{|g''(t_\star,t_\star)|^{\frac 1 2}} \frac{\omega(t_\star,t_\star) t_\star^{2\nu -2}}{(1+t_\star)^{2\nu}} \exp \Big( - \frac 1 h \int_0^{2L} \sqrt{\frac{B^2r^2}{4} + |v_0|} \dd r \Big),
\end{equation}
when $h\to 0$. We insert the estimate \eqref{eq:Ch} on $C_h$, and we find an explicit constant $C(B,L,v) >0$ such that
\begin{equation}
w_h \sim - C(B,L,v) h^{\frac 1 2} \exp \Big( - \frac 1 h \int_0^{2L} \sqrt{\frac{B^2r^2}{4} + |v_0|} \dd r + \frac{2}{h} \tilde{d}(L) - \frac 2 h d(0,L) \Big).
\end{equation}
We find the exponential rate \eqref{form.S} since
\begin{align*}
\int_0^{2L} &\sqrt{\frac{B^2r^2}{4} + |v_0|} \dd r - 2 \tilde d(L) =  \int_0^L \sqrt{\frac{B^2(2L-r)^2}{4} -v_0} - \sqrt{\frac{B^2r^2}{4} - v_0} \dd r.
\end{align*}


\subsection{Condition on the error terms} \label{subsec.fin}

To ensure the error terms from Theorem \ref{thm.split} to be smaller than $w_h$, we need $S< 2 d(0,2L-a)$. We recall formula \eqref{form.S} for $S$, which can be rewritten as
\begin{equation}
S = d(0,2L) + d(0,L) - \tilde d (L) = 2d(0,2L) - \tilde d(2L).
\end{equation}
Thus,
\begin{align*}
S-2d(0,2L-a) = 2d(2L-a,2L) - \tilde d(2L) = d(2L-a,2L) - \tilde d(2L-a).
\end{align*}
Hence the condition $S < 2d(0,2L-a)$ is satisfied as soon as
\begin{equation}\label{condition}
d(2L-a,2L) < \tilde{d}(2L-a).
\end{equation}
We now show that \eqref{condition} is true as soon as $L> (1+\frac{\sqrt 3}{2}) a$. First, we bound the left-hand side by
\begin{equation}
d(2L-a,2L) \leq \int_{2L-a}^{2L} \frac{Br}{2} + \sqrt{|v_0|} \dd r = \sqrt{|v_0|} a + \frac{Ba}{4}(4L-a),
\end{equation}
and the right-hand side
\begin{equation}
\tilde d (2L-a) \geq \int_0^a \sqrt{| v_0 |} \dd r + \int_a^{2L-a} \frac{Br}{2} \dd r = \sqrt{|v_0|} a + BL(L-a).
\end{equation}
Thus, a sufficient condition for \eqref{condition} to hold is
\[B L(L-a) > B La - \frac{Ba^2}{4},\]
which is equivalent to $L > (1+ \frac{\sqrt{3}}{2})a$.

\section*{Acknowledgements}

The author thanks Søren Fournais, Bernard Helffer, Ayman Kachmar, and Nicolas Raymond for many enlightening discussions, and for encouraging this work. 

This work is funded by the European Union. Views and opinions expressed are however those of the author only and do not necessarily reflect those of the European Union or the European Research Council. Neither the European Union nor the granting authority can be held responsible for them.

\bibliographystyle{plain}
\bibliography{bibliotunnel}

\begin{thebibliography}{10}

\bibitem{AA23}
K.~Abou~Alfa.
\newblock Tunneling effect in two dimensions with vanishing magnetic fields.
\newblock {\em arXiv:2212.04289}, 2023.

\bibitem{BHR17}
V.~Bonnaillie-No{\"e}l, F.~H{\'e}rau, and N.~Raymond.
\newblock Semiclassical tunneling and magnetic flux effects on the circle.
\newblock {\em J. Spectr. Theory}, 7(3):771--796, 2017.

\bibitem{BHR22}
V.~Bonnaillie-No{\"e}l, F.~H{\'e}rau, and N.~Raymond.
\newblock Purely magnetic tunneling effect in two dimensions.
\newblock {\em Invent. Math.}, 227(2):745--793, 2022.

\bibitem{DS}
M.~Dimassi and J.~Sj\"{o}strand.
\newblock {\em Spectral asymptotics in the semi-classical limit}, volume 268 of
  {\em London Mathematical Society Lecture Note Series}.
\newblock Cambridge University Press, Cambridge, 1999.

\bibitem{FSW22}
C.~Fefferman, J.~Shapiro, and M.~I. Weinstein.
\newblock Lower bound on quantum tunneling for strong magnetic fields.
\newblock {\em SIAM J. Math. Anal.}, 54(1):1105--1130, 2022.

\bibitem{FHK22}
S.~Fournais, B.~Helffer, and A.~Kachmar.
\newblock Tunneling effect induced by a curved magnetic edge.
\newblock In {\em The physics and mathematics of {E}lliott {L}ieb---the 90th
  anniversary. {V}ol. {I}}, pages 315--350. EMS Press, Berlin, 2022.

\bibitem{FMR}
S.~Fournais, L.~Morin, and N.~Raymond.
\newblock Purely magnetic tunnelling between radial magnetic wells.
\newblock {\em arXiv:2308.04315}, 2023.

\bibitem{Hel88}
B.~Helffer.
\newblock {\em Semi-classical analysis for the {S}chr\"{o}dinger operator and
  applications}, volume 1336 of {\em Lecture Notes in Mathematics}.
\newblock Springer-Verlag, Berlin, 1988.

\bibitem{HK22}
B.~Helffer and A.~Kachmar.
\newblock Quantum tunneling in deep potential wells and strong magnetic field
  revisited.
\newblock {\em arXiv:2208.13030}, 2023.

\bibitem{HKS23}
B.~Helffer, A.~Kachmar, and M.~P. Sundqvist.
\newblock Flux and symmetry effects on quantum tunneling.
\newblock {\em arXiv:2307.06712}, 2023.

\bibitem{HS}
B.~Helffer and J.~Sj\"{o}strand.
\newblock Multiple wells in the semiclassical limit. {I}.
\newblock {\em Comm. Partial Differential Equations}, 9(4):337--408, 1984.

\bibitem{HS85b}
B.~Helffer and J.~Sj{\"o}strand.
\newblock Puits multiples en limite semi-classique. {II}: {Interaction}
  mol{\'e}culaire. {Sym{\'e}tries}. {Perturbation}. ({Multiple} wells in the
  semi-classical limit. {II}: {Molecular} interaction. {Symmetry}.
  {Perturbation}).
\newblock {\em Ann. Inst. Henri Poincar{\'e}, Phys. Th{\'e}or.}, 42:127--212,
  1985.

\bibitem{HS87}
B.~Helffer and J.~Sj{\"o}strand.
\newblock Effet tunnel pour l'{\'e}quation de {Schr{\"o}dinger} avec champ
  magn{\'e}tique. ({Tunnel} effect for the {Schr{\"o}dinger} equation with a
  magnetic field).
\newblock {\em Ann. Sc. Norm. Super. Pisa, Cl. Sci., IV. Ser.}, 14(4):625--657,
  1987.

\bibitem{Nakamura}
S.~Nakamura.
\newblock Tunneling estimates for magnetic {S}chr\"{o}dinger operators.
\newblock {\em Comm. Math. Phys.}, 200(1):25--34, 1999.

\end{thebibliography}

\end{document}